\documentclass[12pt]{amsart}
\usepackage{graphicx}
\usepackage[nocompress]{cite}
\usepackage{amssymb, comment, color}
\usepackage[toc,page]{appendix}
\usepackage{epsf,overpic,amssymb,amscd,times,euscript}
\usepackage[utf8]{inputenc}
\usepackage[english]{babel}
\usepackage{amsmath}
\usepackage{amsthm}
\usepackage{amsfonts}
\usepackage{amssymb}
\usepackage{amscd}
\usepackage{bezier}
\usepackage{enumerate}
\usepackage{mathtools}

\usepackage{graphicx}
\usepackage{upgreek}

\usepackage[colorlinks=true,urlcolor=green,bookmarks=true,bookmarksopen=true,citecolor=green]{hyperref}

\usepackage{parskip}

\usepackage{bbm}

\newcommand{\N}{\mathbbm{N}}                     
\newcommand{\R}{\mathbbm{R}}                     

\newtheorem{thm}{Theorem}[section]               
\newtheorem*{thm*}{Theorem}               
\newtheorem*{cor*}{Corollary}        
\newtheorem{lem}[thm]{Lemma}  
\newtheorem*{lem*}{Lemma}
\theoremstyle{definition}
\newtheorem{defn}[thm]{Definition}      

\newtheorem{rem}[thm]{Remark}           
\newtheorem{question}{Question}
 \newtheorem*{acknowledgement*}{\protect\acknowledgementname}
\newcounter{claim}

 \providecommand{\acknowledgementname}{Acknowledgement}

\author{Marcelo R.R. Alves}
\thanks{M.R.R. Alves was supported by Deutsche
Forschungsgemeinschaft (DFG, German Research Foundation) via the grant “Himmelsmechanik, Hydrodynamik und Turing-Maschinen” - 541525489.}
\address{Marcelo R.R. Alves, 
    Institute of Mathematics ,\\
	University of Augsburg,
	Chair Analysis and Geometry ,
	Universitätsstraße 14,
    DE-86159 Augsburg
	Germany.}
\email{\texttt{marcelorralves@gmail.com}}

\author{Matthias Meiwes}
\address{Matthias Meiwes.}
\email{\texttt{matthias.meiwes@live.de}}

\title{Polytopes and $C^0$-Riemannian metrics with positive $h_{\rm top}$}

\begin{document}

\begin{abstract}
We study Reeb dynamics on starshaped hypersurfaces in $\R^4$ arising as smoothings of starshaped polytopes. Using the $C^0$--stability of positive topological entropy for Reeb flows in dimension three from \cite{ADMP}, we show that there exist starshaped polytopes $P$ such that for any starshaped smoothing of $\partial P$ the associated Reeb flows have positive topological entropy. This answers a question of Ginzburg and Ostrover.

Similarly, we show that given a closed surface $M$ and a number \linebreak $C>0$, there exist continuous and non-differentiable Riemannian metrics $g$ on $S$ with $h_{\rm top}>C$ in the sense that for any smoothing of $g$ the associated geodesic flows have $h_{\rm top}>C$.

\end{abstract}

\maketitle
\tableofcontents

\section{Introduction}

The objective of this note is to study the complexity of Reeb dynamics of polytopes and $C^0$-contact forms and geodesic dynamics $C^0$-Riemannian metrics. These objects lie at the boundary of the  realm of classical dynamics because one cannot in general define Reeb flows of polytopes and $C^0$-contact forms, or geodesic flows of $C^0$-Riemannian metrics. On the other hand, one can always approximate these objects by smooth ones in their respective classes and study the dynamics of these approximations. 
In this note, we show how recent advances in the field of symplectic dynamics can be used to provide a natural and interesting definition of dynamical complexity for these non-smooth objects.

The topological entropy $h_{\rm top}$ is a dynamical invariant, a non-negative number that measures the exponential complexity of a dynamical system. It plays an important role in the theory of dynamical systems, connecting the topological and the measure theoretical approaches to studying dynamics.
Positivity of $h_{\rm top}$ for a dynamical system means that the system possesses some type of topological chaos; we refer the reader to \cite{Hasselblatt-Katok}.

In this note, we prove the existence of polytopes $P$ in $\mathbb{R}^4$ such that the Reeb flows on $\partial P$ have positive topological entropy in the following sense: for any star-shaped smoothing of $P$ the corresponding Reeb flow will have positive topological entropy.

We prove similar results for $C^0$-contact forms on closed contact $3$-manifolds and $C^0$-Riemannian metrics. For the basics of contact geometry and Reeb flows we refer the reader to \cite{GeigesBook}. For the basics of geodesic flows of Riemannian metrics we refer the reader to \cite{Paternain}.

\subsection{Definitions and main results}

By a starshaped smoothing of $\partial P$ we mean a sequence of $C^3$-smooth\footnote{For the main result of this paper we need the smoothings to be $C^3$. Notice that in \cite{ChaidezHutchings2021_ReebPolytopes} the authors consider $C^1$-smoothings.} starshaped hypersufaces in $\R^4$ which converge in the $C^0$-distance to $\partial P$. We say that for a star-shaped smoothing of $\partial P$ its Reeb flow  has positive topological entropy, if the topological entropies of the Reeb flows of sufficiently large elements of the star-shaped smoothing sequence are bounded away from $0$ by a positive constant. 

\begin{defn} \label{def:entropy-polytopes}
Let $(M_j)_{j\in \N }$ be a starshaped smoothing of $\partial P$. Define
\[
h_{\rm top} ((M_j)_{j\in \N}):= \liminf_{j\to +\infty} h_{\rm top}(M_j),
\]
where $h_{\rm top}(M_j)$ denotes the topological entropy of the Reeb flow on $M_j$. We define the topological entropy $h_{\rm top}(\partial P)$ of $\partial P$ as the infimum of the topological entropy of its star-shaped smoothings.
\end{defn}

\begin{rem}
This is, in our opinion, the best possible definition of the topological entropy of the Reeb dynamics on polytopes.

For a  convex or starshaped polytope $P$ one cannot define a Reeb flow on the boundary of $P$ because the characteristic line field on $\partial P$ (induced by the canonical symplectic form $\omega_0$ on $\R^4$) is not continuous at the vertices and some of the edges of the polytope; see for example \cite{ChaidezHutchings2021_ReebPolytopes}. It is thus natural to consider smoothings of $\partial P$ to define the dynamical complexity of $\partial P$.
\end{rem}

Reeb flows on the boundary of polytopes have recently been considered in several works: \cite{ChaidezHutchings2021_ReebPolytopes,HaimKislevOstroverViterbo,HaimKislev2019_symplectic_polytopes,HaimKislevOstrover2023_symplectic_barriers}.

Starshaped polytopes in $\mathbbm{R}^4$ are a particular case of $C^0$-contact forms on the contact manifold $(S^3,\xi_{\rm tight})$. Given a smooth closed cooriented contact $3$-manifold $(M,\xi)$ a $C^0$-contact form $\alpha$ on $(M,\xi)$ is a continuous $1$-form on $M$ such that for every $p \in M$ we have $\ker \alpha_p - \xi_p$, i.e. a continuous defining $1$-form for the $2$-dimensional distribution $\xi$. Since $\alpha$ is only continuous we cannot define its Reeb vector-field and its Reeb flow, so we use smoothings of $\alpha$ to associate to it a dynamical complexity. 

A smoothing of a $C^0$-contact form $\alpha$ on $(M,\xi)$ is a sequence $(\alpha_j)_{j\in \N}$ of $C^3$-smooth contact forms on $(M,\xi)$ such that $\alpha_j$ converges to $\alpha$ in the $C^0$ sense as $j \to +\infty$. Since for each $\alpha_j$ there exists a unique function $f_j: M \to \R$ such that $e^{f_j}\alpha=\alpha_j$, this equivalent to demanding that the sequence of function $f_j$ converges uniformly to the $0$-function on $M$. We then define:
\begin{defn}\label{def:entropy-C0contactforms}
Let $(M,\xi)$ be a smooth closed cooriented contact $3$-manifold and $\alpha$ a $C^0$-contact form $(M,\xi)$.
Given a smoothing $(\alpha_j)_{j\in \N }$ of $\alpha$, we define
\[
h_{\rm top} ((\alpha_j)_{j\in \N}):= \liminf_{j\to +\infty} h_{\rm top}(\phi_{\alpha_j}),
\]
where $\phi_{\alpha_j}$ is the Reeb flow of $\alpha_j$. We then define the topological entropy $h_{\rm top}(\alpha)$ of $\alpha$ as the infimum of the topological entropy of its smoothings.  
\end{defn}

\begin{rem}
This is, in our opinion, the best possible definition of the topological entropy of the Reeb dynamics for $C^0$-contact forms and geodesic dynamics for $C^0$-Riemannian metrics.

A remarkable consequence of the results of \cite{ADMP} is that for a $C^\infty$-open and dense set $\mathcal{U}$ of $C^3$-smooth contact forms on $(M,\xi)$, for any $\alpha \in \mathcal{U}$ the topological  $ h_{\rm top}(\alpha)$ in the sense of Definition \ref{def:entropy-C0contactforms} coincides with the classical definition of the topological entropy $h_{\rm top}(\phi_\alpha)$ of the Reeb flow of $\alpha$. \textbf{In other words, Definition \ref{def:entropy-C0contactforms} gives a lower semicontinuous extension of the classical definition of $h_{\rm top}$ from  $\mathcal{U}$ to the set of $C^0$-contact forms.} Moreover, as conjectured in \cite{ADMP}, we expect that the $\mathcal{U}$ should be the set of all $C^3$-smooth contact forms on $(M,\xi)$.
\end{rem}

In the same spirit, we define the topological entropy of a $C^0$-Riemannian metric on a surface. For us, a  $C^0$-Riemannian metric $g$ on a surface $S$ is a continuous family of inner products on the tangent space of $S$. A smoothing of $g$ is any sequence $(g_j)_{j \in \N}$ of $C^3$-smooth Riemannian metrics such that $g_n$ converges in $C^0$ to $g$. 

\begin{defn}\label{def:entropy-C0-Riemannian-metrics}
Let $g$ be a $C^0$-Riemannian metric on a closed surface $S$. Given a smoothing $(g_j)_{j\in \N }$ of $g$, we define
\[
h_{\rm top} ((g_j)_{j\in \N}):= \liminf_{j\to +\infty} h_{\rm top}(g_j),
\]
where $h_{\rm top}(g_j)$ denotes the topological entropy of the geodesic flow on $g_j$ restricted to the unit sphere bundle. We then define the topological entropy $h_{\rm top}(g)$ of $g$ as the infimum of the topological entropy of its smoothings.
\end{defn}

\begin{rem}
Again, it is a consequence of the results of \cite{ADMP} is that for the set $\mathcal{M}_{\rm bumpy}$ of bumpy $C^3$-smooth Riemannian metrics on $S$, we have that for any $g \in \mathcal{M}_{\rm bumpy}$ the topological  $ h_{\rm top}(g)$ in the sense of Definition \ref{def:entropy-C0-Riemannian-metrics} coincides with the classical definition of the topological entropy $h_{\rm top}(\phi_g)$ of the geodesic flow of $g$. \textbf{In other words, Definition \ref{def:entropy-C0-Riemannian-metrics} provides a lower semicontinuous extension the classical definition of $h_{\rm top}$ from the set of bumpy metrics to the set of $C^0$-Riemannian forms.} Moreover, as conjectured in \cite{ADMP}, we expect that the results in \cite{ADMP} should be valid for all $C^3$-smooth Riemannian metrics on $S$. If this is indeed the case, Definition \ref{def:entropy-C0-Riemannian-metrics} would extend the definition of topological entropy from $C^3$-smooth Riemannian metrics on surfaces to $C^0$-Riemannian metrics on surfaces.

\end{rem}

Our main result regarding polytopes is:
\begin{thm}\label{thm:polytope-positive-entropy}
There exist starshaped polytopes $P\subset \R^4$ such that for any starshaped smoothing $\widetilde M$ of $\partial P$, the Reeb flow on $\widetilde M$ has positive topological entropy.
\end{thm}

This gives a partial answer to a question posed independently by Ginzburg and Ostrover: they asked whether there exist polytopes $P\subset \mathbbm{R}^4$ with the property that the Reeb flow on $\partial P$ has positive $h_{\rm top}$ in the sense above. Our result does not give explicit examples of such polytopes, but suggests that such examples should be abundant. It is a very interesting project to recover a lower bound on the topological entropy of smoothings of $\partial P$ from the combinatorial data of the polytope $P$. See Section \ref{sec:questions}.

For the polytopes $P$ covered by Theorem~\ref{thm:polytope-positive-entropy}, it follows from \cite{CGGM-Reeb,GGM-geodesic,Meiwes-linking} that the Reeb flows on any smooth starshaped smoothing of $\partial P$ exhibit rich and dynamically complicated orbit structures. This complexity is reflected both on the linking structure of the periodic orbits of the flow and on the growth properties of its  Floer–theoretic invariants. 

Our next theorem shows the existence of non-differentiable $C^0$-contact forms on any closed cooriented contact $3$-manifold with arbitrarily large $h_{\rm top}$ in the sense of \ref{def:entropy-C0contactforms}, and non-differentiable $C^0$-Riemannian metrics on closed orientable surfaces with arbitrarily large $h_{\rm top}$. 

\begin{thm} \label{thm:entropy-for-C0contact/metrics}
Let $S$ be a closed surface. For any $C>0$, there exists a non-differentiable $C^0$-Riemannian metric $g_C$ with area $1$ and such that $h_{\rm top}(g_C)>C $.

Let $(M,\xi)$ be a closed cooriented contact $3$-manifold. For any $C>0$, there exists a non-differentiable $C^0$-contact form $\alpha_C$ such that $h_{\rm top}(\alpha_C) > C$.
\end{thm}

Before proving our results, we remark that there are examples of polytopes that do not satisfy the conclusion of Theorem \ref{thm:polytope-positive-entropy}. The simplest such example is the unit cube $[-1,1]^4\subset\R^4$. 

\begin{lem}\label{prop:cube-zero-entropy}
The boundary of the unit cube $B_\infty:=[-1,1]^4\subset\R^4$ admits  a sequence $(\Sigma_p)_{p\in \mathbbm{N}}$ of smooth convex starshaped smoothings whose Reeb flows have vanishing topological entropy.
\end{lem}

\begin{proof}

We consider $\R^4$ endowed with the standard symplectic form
\[
\omega_0 =  dx_1 \wedge dy_1 + dx_2\wedge dy_2.
\]
and the standard Liouville form
\[
\lambda_0 = \tfrac12 \sum_{i=1}^2 (x_i\,dy_i - y_i\,dx_i).
\]
For $p\in\N$, consider the Hamiltonian
\[
H_p(x) \;=\; |x_1|^p + |y_1|^p + |x_2|^p + |y_2|^p,
\qquad x=(x_1,y_1,x_2,y_2)\in\R^4.
\]
For any $p$, the function $H_p$ is smooth on $\R^4$ and strictly convex in $\R^4 $ . Let
\[
\Sigma_p := H_p^{-1}(1).
\]
Then $\Sigma_p$ is a smooth, closed, strictly convex hypersurface. It is also starshaped with respect to the origin.

Let
\[
B_p := \{ x\in\R^4 \mid H_p(x)\le 1\}.
\]
It is well-known that, as $p\to\infty$, the sets $B_p$ converge in the Hausdorff topology to the unit cube
\[
B_\infty := [-1,1]^4.
\]
In particular, the boundaries satisfy
\[
\Sigma_p = \partial B_p \;\xrightarrow[C^0]{}\; \partial([-1,1]^4).
\]

We consider the induced contact form $\alpha_p := \lambda_0|_{\Sigma_p}$.
The Hamiltonian $H_p$ is completely integrable. For example, it admits the integrals $F^1_p(x_1,y_1,x_2,y_2) =  |x_1|^p + |y_1|^p$ and  $F^2_p(x_1,y_1,x_2,y_2) =  |x_2|^p + |y_2|^p$, which are non-degenerate in the sense of \cite{Paternain}. This implies that the Hamiltonian flow $\phi^t_{H_p}$ of $H_p$ on the energy level $\Sigma_p$ has vanishing topological entropy. The same is true for the Reeb flow $\phi^t_{\alpha_p}$ of $\alpha_p$ since this flow is a time reparametrization of  $\phi^t_{H_p}$.




This shows that the boundary of the unit cube $[-1,1]^4$ admits smooth convex starshaped smoothings whose Reeb flows have vanishing topological entropy.
\end{proof}

\section{Proof of the main results}

We start with the proof of Theorem \ref{thm:polytope-positive-entropy}.

\textit{Proof of Theorem \ref{thm:polytope-positive-entropy}.}

We work in $\R^4$ endowed with the standard symplectic form
\[
\omega_0 = d\lambda_0, \qquad 
\lambda_0 = \tfrac12 \sum_{i=1}^2 (x_i\,dy_i - y_i\,dx_i).
\]

By \cite[Theorem 1]{ADMP} there exists a contact form $\alpha$ on $(S^3,\xi_{\rm tight})$ whose Reeb flow has positive topological entropy and such that this property is $C^0$-robust: more precisely, there exists $\varepsilon>0$ such that any contact form $\alpha'$ which is $\varepsilon$-close to $\alpha$ in the $C^0$-distance also has $h_{\rm top} > c$, where $c>0$ is any positive number smaller then $h_{\rm top}(\phi_{\alpha_0})$. Here, $\xi_{\rm tight}$ is the unique tight contact structure in $S^3$. Indeed, examples of contact forms on $(S^3,\xi_{\rm tight})$ with positive $h_{\rm top}$ are well-known, and by the main result in \cite{ADMP} there exist arbitrarily small perturbation of $\alpha_0$ in the $C^\infty$-topology which have $C^0$-robust $h_{\rm top}$. Moreover, since contact forms on $(S^3,\xi_{\rm tight})$ with positive $h_{\rm top}$ are dense in the $C^\infty$-topology \cite{Hryniewicz_generic}, and so can be taken to correspond to convex hypersurfaces in $\mathbb{R}^4$.

Fix then $\alpha$ and $c>0$ as above.
Because $\alpha$ is a contact form on $(S^3,\xi_{\rm tight})$, we can realize it as a smooth starshaped hypersurface $M$ in $\R^4$ satisfying
\[
\alpha = \lambda_0|_M.
\]
Let $U\subset \R^4$ be a sufficiently small tubular neighborhood of $M$. After possibly shrinking $U$, we may identify it with a symplectization neighborhood
\[
U \cong (-\delta,\delta)\times M
\]
with symplectic form $d(e^s\alpha)$, where $s$ denotes the symplectization coordinate.

Any smooth starshaped hypersurface $S\subset U$ transverse to the Liouville vector field $\partial_s$ can be written as the graph of a continuous function $s=f(x)$ over $M$. The induced contact form on $S$ is given by
\[
\alpha_S = e^{f}\alpha.
\]
If $U$ is chosen sufficiently small, the contact form $\alpha_S$ must be $C^0$--close to $\alpha$. By the main result of \cite{ADMP} it follows that the Reeb flow on $S$ has $h_{\rm top}>c$.

Now choose a starshaped polytope $P\subset \R^4$ such that
\[
\partial P \subset U,
\]
and which can be written as the graph of a continuous function over $M$ in the tubular coordinates considered above in $U$.  Such a polytope exists since $U$ is an open neighborhood of $M$.

Let $\widetilde M$ be any starshaped smoothing of $P$. By construction, $\widetilde M\subset U$, and therefore $\widetilde M$ is a starshaped hypersurface whose induced contact form is $C^0$-close to $\alpha$. As shown above, this implies that the Reeb flow on $\widetilde M$ has $h_{\rm top}>c>0$.
\qed

\begin{rem}
If one is interested only in convex polytopes and convex smoothings, a version of Theorem \ref{thm:polytope-positive-entropy} can be obtained in this setting via more classical dynamical methods. This can be obtained in the following way. Let $M$ be a smooth closed convex hypersurface in $\R^4$ whose associated Reeb flow $\phi_M$ has positive $h_{\rm top}$. Combining the results of Katok \cite{Katok} and Nitecki \cite{nitecki}, it follows that there exists $\epsilon>0$ such that for any convex hypersurface $M'$ in $\R^4$ with $d_{C^1}(M',M)<\epsilon$ the associated Reeb flow on $M'$ has $h_{\rm top}$ larger than $\frac{h_{\rm top}(\phi_M)}{2}$. A direct application of Theorem 25.7 in \cite{Rockafellar} implies that there exists $\delta>0$ such that if $M'$ is a smooth convex hypersurface in $\R^4 $ such that $d_{C^0}(M',M)<\delta$ then $d_{C^1}(M',M)< \epsilon$. Let then $P$ be a convex polytope in $\R^4$ such that $d_{C^0}(P,M) < \frac{\delta}{2}$. It is clear that any convex smoothing sequence of $P$ must eventually be $\frac{\delta}{2}$-close to $P$ in $d_{C^0}$, and therefore  ${\delta}$-close to $M$ in $d_{C^0}$, and thus ${\epsilon}$-close to $M$ in $d_{C^1}$. By the discussion above, the Reeb flows of any convex smoothing of $P$ must have topological entropy larger than $\frac{h_{\rm top}(\phi_M)}{2}$.
\end{rem}

We now prove Theorem \ref{thm:entropy-for-C0contact/metrics}.

\textit{Proof of Theorem \ref{thm:entropy-for-C0contact/metrics}.}

We only prove the first statement of the theorem, as the proof of the second statement is similar and easier.

Fix $C>0$.
It is well-known that there exists a $C^\infty$-smooth Riemannian metric $g$ on $S$ with area $1$ and such that $h_{\rm top}(\phi_g)>3C$; a proof of this result can be found in \cite{ADMM}. Since, by results of Yomdin and Newhouse, $h_{\rm top}$ is continuous in the $C^\infty$-topology, and bumpy metrics are $C^\infty$-dense in the set of smooth Riemannian metrics, we can take $g$ with area $1$ and such that $h_{\rm top}(\phi_g)>3C$ to be a bumpy metric.

Given Riemannian metrics $g_1,g_2$ on $S$, we write $g_1<g_2$ if for all $v \in TS$ we have $g_1(v,v)<g_2(v,v)$.
Theorem 2 in \cite{ADMP} implies that there exists $\delta>0$ such that for any $C^3$-smooth Riemannian metric $g'$ such that 
\begin{equation} \label{eq:comparisonmetrics}
e^{-2\delta}g < g' < e^{2\delta}g,
\end{equation}
the geodesic flow $\phi_{g'}$ satisfies $h_{\rm top}(\phi_{g'})>C$. Moreover, by possibly choosing a smaller $\delta>0$, we can also assume that the area of any $g'$ satisfying \eqref{eq:comparisonmetrics} is between $\frac{3}{4}$ and $\frac{5}{4}$.

Since \eqref{eq:comparisonmetrics} is an open condition, it is easy to construct a non-differentiable $C^0$-Riemannian metric $g_C$ such that 
\begin{equation} 
e^{-\delta}g < g_C < e^{\delta}g,
\end{equation}
and $g_C$ has area $1$. Since nowhere differentiable functions are $C^0$-dense in the space of continuous functions, we can take $g_C$ to be nowhere differentiable, by ensuring that the functions defining $g_C$ in coordinate charts are nowhere differentiable. 
By the above paragraph, any smoothing sequence of $g_C$ will have $h_{\rm top}$ larger than $C$, since sufficiently large elements of the smoothing sequence will satisfy \eqref{eq:comparisonmetrics}. This shows that $h_{\rm top}(g_C)>C$, as desired.
\qed

\section{Questions} \label{sec:questions}

In this section, we list some questions that we consider interesting for future investigation.

In his Master’s thesis \cite{Tsodikovich}, Tsodikovich studies the topological entropy of the Reeb dynamics on the standard simplex in $\mathbb{R}^4$ from a different perspective. Rather than performing smoothings, he analyzes the Reeb dynamics directly on the polytope by restricting it to a subset on which the dynamics is continuous. Using this approach, he shows that the resulting notion of topological entropy vanishes. A natural and intriguing question is whether the notion of topological entropy considered in \cite{Tsodikovich} coincides with that of Definition~\ref{def:entropy-polytopes}. More generally, we propose the following question:
\begin{question}
Can we recover the topological entropy of a starshaped polytope $P \subset \R^4$ in the sense of Definition \ref{def:entropy-polytopes} from the combinatorial data defining $P$?
Which values in $[0,+\infty]$ can be the topological entropy of polytopes? Is it true that ``generic'' polytopes in $\R^4$ have positive $h_{\rm top}$?
\end{question}

\begin{question}
Which values in $[0,+\infty]$ can be realized as the topological entropy of non-differentiable $C^0$-contact forms on a closed  contact $3$-manifolds?
Which values in $[0,+\infty]$ can be realized as the topological entropy of non-differentiable $C^0$-Riemannian metrics on a closed orientable surface?
\end{question}

\begin{question}
Can the results in \cite{Alves-Meiwes24} be used to extend the notion of topological entropy to the $d_{\rm Hofer}$-completion of the group of Hamiltonian diffeomorphisms of a closed surface? How does this connect to the recent results in $C^0$-symplectic topology in the sense of \cite{BHS,HLS,Opshtein}?
\end{question}

\subsection{Acknowledgments}
We thank Viktor Ginzburg, Yaron Ostrover, and Felix Schlenk for their interest in this work and for their comments and suggestions. We thank our collaborators Lucas Dahinden and Abror Pirnapasov: this note builds on our joint work. We thank Alberto Abbondandolo and Jungsoo Kang for helpful discussions on convex energy levels. 

\bibliographystyle{plain}
\bibliography{ReferencesS5}

\end{document}